\documentclass[14pt]{article}
\usepackage[cp1251]{inputenc}
\usepackage[english]{babel}
\usepackage{amsfonts,amssymb,amsmath,amstext,amsbsy,amsopn,amscd,amsthm,graphicx,euscript,graphicx}
\usepackage{graphics}
\newtheorem{thm}{Theorem}

\newcounter{tdfn}
\newcounter{trk}
\newenvironment{rk}
{\vspace{0.15cm}{\bf Remark \arabic{trk}.}} {\par
\addtocounter{trk}{1}}
{\endtrivlist}

\def\:{\colon}

\def\R{{\mathbb R}}

\def\0{{\mathbf 0}}
\def\1{{\mathbf 1}}

\def\R{{\mathbb R}}

\title{The Groups $G_{n}^{k}$ and Fundamental Groups of Configuration Spaces}

\author{V.O.Manturov \footnote{Bauman Moscow State Technical University and Chelyabinsk State University \newline Research is carried out with the support of Russian Science Foundation (project no. 16-11-10291).}}

\date{}

 \def\R{{\mathbb R}}

\begin{document}
% \magstep2

\maketitle

AMS MSC 57M25,57M27

{\Large

\begin{abstract}
We construct a map from fundamental groups of complements to some plane configurations to the groups $G_{n}^{k}$ for large $k$.
\end{abstract}

\section{Introduction. Basic Notions}
In the paper \cite{Great}, the author introduced a family of groups $G_{n}^{k}$ depending on two positive integers $n>k$,  and formulated the following principle:
{\em if a dynamical system describing a motion of  $n$ particles admits some nice general position codimension $1$ property governed exactly by  $k$ particles, then this dynamical system has invariants valued in $G_{n}^{k}$.}

In \cite{MN}, two partial cases were calculated explicitly. If we consider a motion of $n$ pairwise distinct points on the plane and choose the property  ``some three points are collinear'', then we get a homomorphism from the pure $n$-strand braid group $PB_{n}$ to the group $G_{n}^{3}$. If we choose the property ``some four points belong to the same circle or line'', we shall get a map from $PB_{n}$ to $G_{n}^{4}$. The approach with $G_{n}^{3}$ has been upgraded in \cite{Gn3toGn2}.

In the present paper we construct maps from fundamental groups of configuration spaces to the groups $G_{n}^{k}$.

We are interested in the configuration spaces $C_{n}(\R^{k-1})$ of ordered sets of  $n$ points in $(k-1)$-dimensional space $(n>k)$. For $k=3$ the fundamental group of this space $\pi_{1}(C_{n}(\R^{2}))$ is precisely the Artin pure braid group. For $k>3$ these spaces are simply connected.

Thus we shall consider spaces $C'_{n}(\R^{k-1})$ defined as follows: a point in $C'_{n}(\R^{k-1})$ is a set of $n$ distinct points in  $\R^{k-1}$, where every $(k-1)$ points are in general position. Thus, for example, for $k=3$, the only condition is that no two points among the given $n$ points coincide. In other words, $C'_{n}(\R^{2})=C_{n}(\R^{2})$. For $k=4$, for points $x_{1},\dots, x_{n}$ in three-space we require that no three points are collinear (though some four points can belong to the same plane), for $k=5$ a point in $C'_{n}(\R^{4})$ is an $n$-tuple of points in  $\R^{4}$ with no four of them belonging to the same $2$-plane, etc.
Quite analogously, one defines $C_{n}({\R}P^{k-1})$ and $C'_{n}({\R}P^{k-1})$.

Following  \cite{Great}, let us define the groups $G_{n}^{k}$. For positive integers $n>k$, the group $G_{n}^{k}$ has ${n \choose k}$ generators $a_{m}$, where $m$ runs all unordered sets of  $k$ distinct indices from $1$ to $n$.

$$G_{n}^{k}=\langle a_{m}|(1),(2),(3)\rangle,$$
where  (1) means that $(a_{m})^{2}=1$ for each unordered set \\ $m\subset \{1,\dots,n\}, Card(m)=k$; (2) means that $a_{m}a_{m'}=a_{m'}a_{m},$ whenever $Card(m\cap m')<k-1$; finally,
(3) looks as follows. For every set $U\subset \{1,\dots, n\}$ of cardinality $k+1$, let us order arbitrarily all its $k$-subsets: $m^{1},\dots, m^{k+1}$. Then
$(a_{m^1}\cdots a_{m^{k+1}})^{2}=1$.
This relation can also be rewritten as $a_{m^{1}}\cdots a_{m^{k+1}}=a_{m^{k+1}}\cdots a_{m^{1}}$.

Note that for  $n=k+1$ the relations of type (2) never happen, since every two distinct subsets $m,m'$ of cardinality $k$ of $\{1,\dots, k+1\}$ have intersection of cardinality  $k-1$.

Let us fix  two positive integers $n>k$. Let us consider the space $C'_{n}(\R^{k-1})$. A point $x\in C'_{n}(\R^{k-1})$ is {\em singular}, if the set of points $x=(x_{1},\dots, x_{n})$ representing it contains some  $k$ points which are not in general position, i.e., belong to the same  $(k-2)$-plane. Let us fix two nonsingular points  $x,x'\in C'_{n}(\R^{k-1})$.

Let us consider the set of smooth paths $\gamma_{x,x'}: [0,1]\to C'_{n}(\R^{k-1})$. For every such path we have a set of values $t$ for which the set of points  corresponding to $\gamma_{x,x'}(t)$ is not in general position (some $k$ points belong to the same $k-2$-plane). We call these values $t\in [0,1]$ {\em singular}. We say that a smooth path is {\em good and stable} if the following conditions hold:

\begin{enumerate}

\item The set of singular values $t$ is finite;

\item For each singular value $t=t_{l}$, among the $n$ points representing $\gamma_{x,x'}(t_{l})$ there exists exactly one subset of  $k$ points belonging to the same $(k-2)$-plane;

\item A smooth path is {\em stable}, i.e., the number of singular values does not change under a small perturbation.

\end{enumerate}

We say that two paths $\gamma,\gamma'$ with the same endpoints $x,x'$ are  {\em isotopic}, if there exists a continuous family $\gamma^{s}_{x,x'},s \in [0,1]$, of smooth paths with endpoints fixed, such that $\gamma^{0}_{x,x'}=\gamma,\gamma^{1}_{x,x'}=\gamma'$.
A smooth path is a  {\em braid}, if $x$ and $x'$ correspond to the same set of points (possibly, up to a permutation). A smooth path is a  {\em pure braid} if  $x$ and $x'$ coincide as ordered sets.
We can make every smooth path good and stable with the the ends fixed by an arbitrarily small perturbation.
A braid can also be {\em good and stable}; if it is not so, we can make it good and stable by an arbitrarily small perturbation.

We shall consider the set of smooth paths up to isotopy. For paths $\gamma_{x,x'}$ and $\gamma'_{x',x''}$ the {\em concatenation} operation is well defined; one should consider the path $\gamma''_{x,x''}$, such that $\gamma''(t)=\gamma(2t)$ for $t\in [0,\frac{1}{2}]$ and $\gamma''(t)=\gamma'(2t-1)$ for $t\in [\frac{1}{2},1]$ and smooth it in the neighbourhood of $t=\frac{1}{2}$; this smoothing is unique up to isotopy.

Hence, the set of braids and the set of pure braids (for a fixed $x$) admits a group structure. The latter group is, obviously, isomorphic to the fundamental group of the configuration space $\pi_{1}(C'_{n}(\R^{k-1}))$. The former group is isomorphic to the fundamental group of the quotient space of this configuration space by the obvious action of the permutation group.

\section{The Main Theorem}

Now with each good and stable path we associate an element of the group $G_{n}^{k}$ as follows. Let us enumerate all singular values $0<t_{1}<\dots <t_{l}<1$ of our smooth path (it is agreed that $0$ and $1$ are non-singular). For each singular value $t_{p}$ there exists exactly one unordered set $m_{p}$ of $k$ indices corresponding to this value. Let us associate  $a_{m_{p}}$ to it. With the whole path $\gamma$ we associate the product $f(\gamma)=a_{m_{1}}\cdots a_{m_{l}}$.

\begin{thm}
The map  $f$ described above takes isotopic paths $\gamma_{x,x'},\gamma'_{x,x'}$ with the endpoints fixed ($x,x'$ are
non-singular) to the same element of the group $G_{n}^{k}$.
For pure braids, the map  $f$ is a homomorphism \\ $f:\pi_{1}(C'_{n}(\R^{k-1}))\to G_{n}^{k}$.

\label{th1}
\end{thm}

\begin{proof}
The proof follows from the basic principle from \cite{Great} that in order to consider isotopy between two paths, it suffices to take into account only singularities of codimension at most two (and not to count higher codimension singularities). As singularities of codimension one ($k$-tuples of points on the same $(k-2)$-plane) give rise to generators, relators correspond to singularities of codimension two. Let us list them explicitly.

\begin{enumerate}
\item A non-stable set of $k$ points on a $(k-2)$-plane which disappears after a small perturbation; this corresponds to the relation ${a_{m}}^{2}=1$;

\item Coincidence of two moments when two independent $k$-tuples appear. This corresponds to the relations
$a_{p}a_{q}=a_{q}a_{p}$, where the sets $p,q$ have no more than  $k-2$ common indices;

\item The situation when some $k+1$ points belong to the same $(k-2)$-plane. This path is not good.

     We can make it good by a small perturbation; the resulting good path will contain some $k+1$ moments when each  $k+1$ subsets of $k$ points belong to the same plane \cite{Great}. The ``opposite'' small perturbation leads to the opposite order of these $k+1$ non-stable subsets, which follows from the general position argument from \cite{Great}. This leads to the relation (3).
\end{enumerate}

\end{proof}

\begin{rk}
In the case $k=3$, the map described above coincides with the map from the pure braid group to  $G_{n}^{3}$,  see \cite{MN}.
\end{rk}

The same arguments as above lead to the following
\begin{thm}
Theorem \ref{th1} holds if we replace $\R^{k-1}$ with ${\R}P^{k-1}$.\label{th2}
\end{thm}

\section{Hierarchy, Duality, and Open Questions}

The groups $G_{n}^{k}$ have a hierarchy $\cdots \subset G_{n}^{k}\subset G_{n+1}^{k+1}\subset G_{n+2}^{k+2}\subset\cdots$, where the monomorphism $G_{n}^{k}\to G_{n+1}^{k+1}$ is given by the formula $a_{m}\mapsto a_{m\cup \{n+1\}}$; the inverse map (projection) is given by the formula $a_{m}\mapsto 1,\mbox{ if } (n+1) \notin m, a_{m}\mapsto a_{m\backslash \{n+1\}}, \mbox{ if } (n+1)\in m$.

This hierarchy agrees with embeddings of (projective) spaces and the maps of fundamental groups to $G_{n}^{k}$ described above, as follows.

With each dynamical system of  $n$ pairwise distinct points in ${\R}P^{k-1}$
we can associate a dynamical system of  $n+1$ points in ${\R}P^{k}$ with homogeneous coordinates $(y_{1}:y_{2}:\dots: y_{k}:y_{k+1})$, such that the \\ $(n+1)$-th point is fixed with coordinates $(0:0:\dots 0:1)$, and $j$-th point $j=1,\cdots, n$, has the same coordinates as those of the initial dynamics, and the last coordinate is equal to some constant $s_{j}$. Whatever  coordinate $s_{j}$ we choose, any singular value of the initial dynamical system leads to  a singular value for the resulting dynamical system: $k$ points on a $(k-2)$-dimensional plane together with the $(n+1)$-th point will lie on a $(k-1)$-plane.

Our goal is to choose all $s_{j}$ in such a way that no other singular values occur for the new dynamics.  This can be done by consequently choosing all $s_{j}$ positive and large enough in comparison with all previous $s_{1},\cdots, s_{j-1}$. Namely, the condition that some $k+1$ points in ${\R}P^{k}$ lie on the same $(k-1)$-plane means that some determinant equals zero. The determinant in question will be equal to the sum $D=\pm s_{m_1}D_{1}\pm  s_{m_2}D_{2}\pm \cdots \pm s_{m}D_{m}$ where $D_{j}$ are some minors, which are all non-zero because no $(k-1)$ points belong to the same $(k-3)$-plane. By using the compactness argument, all $D_{i}$ attain their minimum and maximum absolute values, so we can choose all $s_{j}$ such that the sign of $D$ is the same as that of the ``leading term'' $\pm s_{m} D_{m}$ if $s_{m}$ is large enough. Now, the compactness argument can be used not only for a single braid but also for an isotopy between two braids.
In this case, the word corresponding to the new dynamics will be obtained from the word corresponding to the initial dynamics by replacing each $a_{m},$ ($m\subset \{1,\cdots, n\}$) with $a_{m\cup \{(n+1)\}}$, hence the resulting word will be the image of the inclusion map $G_{n}^{k}\to G_{n+1}^{k+1}$.

Studying the image of the group $\pi_{1}(C'_{n}(\R^{k-1}))$ in the group $G_{n}^{k}$ is quite an interesting problem. This map is neither injective nor surjective (we can see that it is non-injective even for $k=3$: the pure braid corresponding to the cyclic rotation of the circle containing all $n$ points, is mapped to the unit element of the group $G_{n}^{k}$ because no three points belonging to the same circle are collinear). This map can be made injective if we consider  a bit more complicated group than  $G_{n}^{3}$, see \cite{Imaginaire}.

As for the image $Im f$, the following holds.
Let us fix some \\ $k-1$ indices $m_{1},\dots, m_{k-1}\in \{1,\dots,n\}$, and order the remaining indices: $l_{1},\dots, l_{n-k+1}$. We say that an element of $G_{n}^{k}$ is {\em elementary} (or $\epsilon$-type) element, if it is equal to
$(a_{m_{1},\dots, m_{k-1}, l_{1}}\cdots, a_{m_{1},\dots m_{k-1}, l_{n-k+1}})^{2}$.

\begin{thm}
For $k>3$ every element from $Im f\in G_{n}^{k}$ is a product of $\epsilon$-type elements and their conjugates.
\end{thm}

We can sketch the proof as follows (we illustrate it for $k=4$; for larger $k$ it is analogous). Let some point $x_{3}$ lie in a small $\delta$-neighbourhood of a line passing through the points $x_{1},x_{2}$ in three-space. Let us rotate $x_{3}$ around this line. Then the plane passing through $x_{1},x_{2},x_{3}$ will rotate together with $x_{3}$; after the $2\pi$-rotation, for each of the points $x_{j},j\neq 1,2,3,$ there will be two moments when $x_{j}$ belongs to the plane passing through $x_{1},x_{2},x_{3}$; these two moments  differ by a  $\pi$-rotation. If we choose $\delta$ small enough, there will be no other singular moments; this means that we get exactly the square of some word of length $(n-3)$, i.e., an $\epsilon$-type element.
It is clear that instead of $1,2,3$ we can consider any three indices (for $k=4$).

The fact that the image of $Im f$ consists only of products of conjugates to $\epsilon$-type elements, can be seen as follows. The space $C_{n}(\R^{k-1})$ is simply connected. Hence, in order to understand the fundamental group of the space $C'_{n}(\R^{k-1})$ one should see what happens in the neighbourhood of the complement $C_{n}(\R^{k-1})\backslash C'_{n}(\R^{k-1})$, i.e., in the neighbourhood of sets of those points where some $k-1$ points belong to the same  $(k-3)$-plane. The elements of $G_{n}^{k}$ corresponding to cycles around such deleted subsets are exactly $\epsilon$-type elements. Thus, with a reference point $A$ fixed, we can decompose each path in $C'_{n}(\R^{k-1})$ into a collection of paths of the following type. Each path approaches complement $C_{n}(\R^{k-1})\backslash C'_{n}(\R^{k-1})$, rotates around this complement, and returns back to $A$. This precisely means that we have a product of conjugates of $\epsilon$-type elements

Certainly, this argument fails for $\pi_{1}(C'_{n}({\R}P^{k-1}))$ since real projective spaces are not simply connected.

According to the general principle of \cite{Great}, the particle should not necessarily be a point.
For example one can treat Theorems \ref{th1},\ref{th2} by using projective duality.

Namely, let us consider the space of sets of $n$ (projective) hyperplanes in $k$-dimensional real (projective) space, where every $k-1$ hyperplanes are in general position. The problem of studying of this space is exactly the problem of studying  $C'_{n}(\R^{k})$ (resp., $C'_{n}({\R}P^{k})$).

This naturally extends the application domain of groups  $G_{n}^{k}$ to the study of various subspaces of various codimensions.

It would be extremely interesting to establish the connection between $G_{n}^{k}$ with the Manin-Schechtmann ``higher braid groups'' \cite{ManinSchechtmann}, where the authors study the fundamental group of complements to some configurations of {\em complex hyperplanes}. One of the relations in the paper of M.Kapranov and V.Voevodsky \cite{KapranovVoevodsky} looks quite similar to our relation (3) for $G_{n}^{k}$.

Besides the questions mentioned above, the following question seems quite interesting to the author: how to realize $G_{n}^{k}$ for $k\ge 5$, by a motion of points on $\R^{2}$; this will lead to a homomorphism of the pure braid group to $G_{n}^{k},k\ge 5$. As mentioned above, there is a homomorphism from $PB_{n}$ to $G_{n}^{4}$.

It seems likely that the connections between the Artin braid group and all such geometric realizations of groups $G_{n}^{k}$ can be achieved in the following manner. Take an open disc in $\R^{2}$ and embed it smoothly into $\R^{k-1}$ in such a way that
no $(k-1)$ points on the embedded surface belong to the same $(k-3)$-plane. Then we obviously get a map from the pure braid group to $G_{n}^{k}$. For $k=4$ we can map the disc $\{x,y|x^{2}+y^{2}<1\}$ to the half-sphere: $(x,y)\mapsto(x,y,\sqrt{1-x^{2}-y^{2}})$ which would lead
to a nice picture on the half-sphere: no three points on the half-sphere are collinear, and if four points on the sphere belong to
the same plane, then they belong to the same circle, cf. with \cite{MN}.

So, the following natural question arises: which topological spaces $T$ can be (smoothly) embedded into $\R^{k-1}$ in such a way that no $k-1$ points in the image of $T$ belong to the same $(k-1)$ plane?

I am grateful to D.P.Ilyutko and D.A.Fedoseev for their useful comments.
}

 \end{document}